\documentclass[12pt]{article}

\usepackage{amsthm}
\usepackage[english]{babel}
\usepackage[utf8]{inputenc}
\usepackage{amsmath}
\usepackage{amsthm}
\usepackage{graphicx}
\usepackage[colorinlistoftodos]{todonotes}
\usepackage{wasysym}
\usepackage{amsfonts}
\usepackage{ amssymb }
\usepackage{tabu}
\usepackage{mathtools}
\usepackage{esvect}
\usepackage{amsmath,graphicx}
\usepackage{amsmath,lipsum}
\usepackage{enumitem}
\usepackage{physics}
\usepackage{tikz}
\usetikzlibrary{trees}
\usepackage{centernot}
\usepackage{extsizes}

\theoremstyle{plain}
\usepackage{geometry}
\title{On the Expected Value of the Maximal Bet in the Labouchere System}

\author{Nina Zubrilina}

\date{\today}

\begin{document}

\maketitle
\theoremstyle{definition}
\newtheorem{theorem}{Theorem}[section]
\newtheorem{observation}[theorem]{Observation}
\newtheorem{corollary}[theorem]{Corollary}
\newtheorem{lemma}[theorem]{Lemma}
\newtheorem{lemma*}{Lemma}
\newtheorem{defn}[theorem]{Definition}
\newtheorem{remark}[theorem]{Remark}

\newtheorem*{probaux}{Problem \protect\probnumber}
\newenvironment{prob}[1]{\def\probnumber{#1}\probaux \mbox{}\\}{\endprobaux}

\newenvironment{letter}{\begin{enumerate}[a)]}{\end{enumerate}}

\newlist{num}{enumerate}{1}
\setlist[num, 1]{label = (\alph*)}
\newcommand{\myitem}{\item}

\newcommand{\Prob}{\mathbb{P}}
\newcommand{\E}{\mathbb{E}}
\newcommand{\Q}{\mathbb{Q}}
\newcommand{\R}{\mathbb{R}}
\newcommand{\N}{\mathbb{N}}
\newcommand{\Z}{\mathbb{Z}}
\newcommand{\F}{\mathbb{F}}
\newcommand{\C}{\mathbb{C}}
\renewcommand{\F}{\mathbb{F}}
\newcommand{\Cl}{\operatorname{Cl}}
\newcommand{\disc}{\operatorname{disc}}
\newcommand{\Gal}{\operatorname{Gal}}
\newcommand{\Img}{\operatorname{Im}}
\newcommand{\eps}{\varepsilon}

\newcommand\ang[1]{\left\langle#1\right\rangle}

\renewcommand{\d}{\partial}
\newcommand*\Lapl{\mathop{}\!\mathbin\bigtriangleup}
\newcommand{\mat}[4] {\left(\begin{array}{cccc}#1 & #2\\ #3 & #4 \end{array}\right)}
\newcommand{\mattwo}[2] {\left(\begin{array}{cc}#1\\ #2 \end{array}\right)}
\newcommand{\pphi}{\varphi}
\newcommand{\eqmod}[1]{\overset{\bmod #1}{\equiv}}
\renewcommand{\O}{\mathcal{O}}
\newcommand{\fr}[1]{\mathfrak{#1}}
\newcommand{\p}{\mathfrak{p}}
\newcommand{\m}{\mathfrak{m}}
\newcommand{\limto}[1]{\xrightarrow[#1]{}}

\newcommand{\under}[2]{\mathrel{\mathop{#2}\limits_{#1}}}

\newcommand{\underwithbrace}[2]{  \makebox[0pt][l]{$\smash{\underbrace{\phantom{%
    \begin{matrix}#2\end{matrix}}}_{\text{$#1$}}}$}#2}

\begin{abstract}
Consider a game consisting of independent turns with even money payoffs in which the player wins with a fixed probability $p \geq 1/3$ and loses with probability $1 - p$. The Labouchere system is a betting strategy which entails keeping a list of positive real numbers and betting the sum of the first and the last number on the list at every turn. In case of a victory, those two numbers are erased from the list, and, in case of a loss, the bet amount is appended to the end of the list. The player finishes the game when the list becomes empty. It is known that, in a game played with the Labouchere system with $p \leq 1/2$, both the sum of the bets and the maximal deficit have infinite expectation. Grimmett and Stirzaker raised the question of whether the same is true for the maximal bet. In this paper we show the expectation of the maximal bet is finite for $p > c$, where $c \approx 0.613763$ solves $(c-1)^2 c = \frac{8}{27(1+\sqrt{5})}$, thus partially answering this question. 
\end{abstract}

\section{Introduction}
Let $L = (\ell_1, \ldots, \ell_j)$ with $\ell_i \in \R_{> 0} $ be a list of length $j$ and let $b(L)$ be the sum of the first and last terms in the list, or, more specifically, 
\[
b(L) := 
\begin{cases}
0 &\text{ if } j = 0,\\
\ell_1&\text{ if } j = 1,\\
\ell_1 + \ell_j&\text{ if } j\geq 2.
\end{cases}
\]
Let $X_1, X_2, \ldots$ be i.i.d.\hspace{0.03in}random variables with $P(X_i = 1) = p \in [1/3, 1], \ P(X_i = -1) = q := 1-p.$ Given a list $L = (\ell_1, \ldots, \ell_j)$, the list $L[X_i]$ is constructed as follows: \\

\[
L[1]  : = 
\begin{cases}
\emptyset &\text{ if } j \leq 2;\\
(\ell_2, \ldots, \ell_{j- 1}) &\text{ if } j \geq 3;\\
\end{cases}
\hspace{0.16in}
\]
\[
L[-1]: = 
\begin{cases}
\emptyset &\text{ if }  j =0;\\
(\ell_1, \ldots, \ell_{j}, b(L)) &\text{ if }  j \geq 1;\\
\end{cases}
\]
This construction is used in the \emph{Labouchere system}\footnote{Henry Labouchere is most remembered for the Labouchere Amendment, which was used to prosecute Oscar Wilde, Alan Turing, and many others.} --- a betting strategy for games where at every turn, we win with probability $p$ and lose with probability $q  = 1-p$. Starting with a list $L$ as above, bet the sum of the first and the last entries of the list, i.e. $b(L)$. If we win, we erase the first and last entry from the list, and, if we lose, we append the lost amount to the end of the list.

For a given game (i.e. sequence $X_1, X_2, \ldots $ with $X_i \in \{-1, 1\})$ with a starting list $L_0$, let $L_1, \ldots, L_N$ be the sequence of lists appearing in this game, $L_{i + 1} := L_i[X_{i+1}]$. Here the \emph{stopping time} $N$ is the number of turns taken before the list becomes empty. For $p \in [1/3, 1]$, $N$ is finite with probability $1$ since $-2p + q \leq 0$. 

It has been shown by Grimmett and Stirzaker \cite[Problem 12.9.15]{grimmett}, that for $p \in (1/3, 1/2]$, $E[N] < \infty$, and for $p \in [1/3; 1/2]$, the total bet amount and the maximal deficit of the player have infinite expectations, i.e.
$$ E(b(L_0) + \ldots + b(L_{N-1}))=  \infty,  \ E(\inf_m X_1 b(L_0) + \ldots +  X_m b(L_m)) = - \infty.$$
They also made an unproved assertion that $M := \sup_m b(L_m)$ has infinite expectation as well \cite[Problem 12.9.15]{grimmettq}. Motivated by this assertion, Ethier conjectured in private communication that $E[M] = \infty$ for $p \in [1/3; 1/2]$ and $E[M] < \infty$ for $p \in (1/2; 1]$. The main objective of this paper is to address the second half of this conjecture.

We approach the question as follows. In section $2$, we summarize some previous results about the dependence of the stopping time $N$ on the initial list. In section $3$, we bound $M$ from above by a function of $N$ for games starting with the list $(1, 2)$. Lastly, combining the two results, in section $4$ we conclude that $E[M]$ is finite for $p \in (c; 1]$, where $c$ is a root of $(x-1)^2x = \frac{8}{27(1 + \sqrt{5})}$, $c \approx 0.613763$ for the starting list $(1, 2)$ and show that the question of finiteness of $E[M]$ for an arbitrary starting list $L$ reduces to the same question for the list $(1, 2)$, thus concluding the proof of the main Theorem: 
\begin{theorem}\label{main}
Let $p \in (c; 1]$, where $c$ is the root of $(x-1)^2x = \frac{8}{27(1 + \sqrt{5})}$ lying in $(1/3, 1]$, $c \approx 0.613763$. Then 
$$E[M] < \infty.$$ 

\end{theorem}

Notation: By $f(x) \ll g(x),$ we mean that there exist $C, X$, such that for all $x > X,$ $\abs{f(x)} < C\abs{g(x)}.$

\section{Random walks on $\Z_+$}

\hspace{.22in}The lengths of the lists that appear in a Labouchere game can be modeled by a random walk in $\Z_+$ with initial state $ j$ and absorbing state $0$ with transition probabilities 
\[
P(h, k) = 
\begin{cases}
p &\text{ if } k = \max(h - 2, 0), \\
 q&\text{ if } k = h + 1, \hspace{1in}  \\
0 &\text{ otherwise}
\end{cases}
\hspace{0.2in} h \geq 1.
\]
The position of the walk after the $n^\text{th}$ step corresponds to the length of the list after $n$ turns. In the same way as before, we define the stopping time $N$ to be the number of steps before getting to the absorbing state $0$. 

We will use $P_{j}$ to denote probabilities conditional on starting the walk at the point $j \in \Z_{>0}$.
We will use a part of a Theorem of Ethier deduced from certain extensions of the ballot theorem:
\begin{theorem}[{{\cite[Theorem 2]{ethier}}}]
$$P_1(N = 3m + 1) = \frac{1}{2m + 1} {3m \choose m} p^{m + 1} q^{2m}\text{ for } m \geq 0,$$
$$P_1(N = 3m + 2) = \frac{1}{m + 1} {3m + 1 \choose m} p^{m + 1} q^{2m + 1}\text{ for } m \geq 0$$ (and $P_1(N = 3m) = 0 $ for $m \geq 1$).
\end{theorem}

Since we will be primarily working with a starting list of length $2$, we state a simple corollary:
\begin{corollary}\label{lengthtwo}

$$P_2(N = 3m ) = \frac{1}{2m + 1} {3m \choose m} p^{m + 1} q^{2m-1} \text{ for } m \geq 1,$$
$$P_2(N = 3m + 1) = \frac{1}{m + 1} {3m + 1 \choose m} p^{m + 1} q^{2m}\text{ for } m \geq 0$$ (and $P_1(N = 3m + 2) = 0 $ for $m \geq 0$).

\end{corollary}
\begin{proof}
For $n > 0$, the number of walks ending after $n$ moves starting at $2$ is equal to the number of walks starting at $1$ and finishing in $n + 1$ moves.
\end{proof}

\section{Bounds on the Maximal Bet in Terms of the Stopping Time}
\hspace{.22in} In this section, we bound the maximal bet in a game with a starting list $L$ from above in terms of the stopping time. We only do it for the games starting on the list $L = (1, 2)$, because this case will imply all the other cases, as shown in Section $4$. 

We denote a winning turn with \textbf{W} and a losing one with \textbf{L}, so a specific game can be described with a sequence of \textbf{W}s and \textbf{L}s.

\begin{defn}
Let $s = s_1\ldots s_k$ with $s_i \in \{\textbf{W}, \textbf{L}\}$ be a sequence of wins and losses. For all $m \leq k$, we say $s': = s_1\ldots s_m$ is a \emph{prefix} of $s$, and if $m < k$, we say $s'$ is a \emph{proper prefix} of $s$.
\end{defn}
\begin{defn}
Let $L$ be a starting list of length $j$ and let $s$ be a sequence of wins and losses. We say $s$ is \emph{playable} on $L$ if for any proper prefix $s' = s_1\ldots s_m$ of $s$, we have 
$$j + \# \abs{\{ i \leq m \ \vert s_i = \textbf{L}\}} - 2 \# \abs{\{ i \leq m \ \vert s_i = \textbf{W}\}}>0,$$ i.e.\ if the series of wins and losses encoded by $s$ can be fully played by a player with a starting list $L$. 
\end{defn}

\begin{lemma}\label{basics}
Let $L = (\ell_1, \ldots, \ell_j)$ and $K = (k_1, \ldots, k_{h})$ be two lists, and let $s$ be a sequence playable on both $K$ and $L$. Let $L' = (\ell_1', \ldots, \ell_{j'}')$ and $K' = (k_1', \ldots, k'_{h'}) $ be the states of the two lists after the sequence $s$ has been played. Then:
\begin{enumerate}
\item If $$0 \leq \ell_1 \leq \ldots \leq \ell_j,$$ then $$0 \leq \ell_1' \leq \ldots \leq \ell_{j'}';$$ 
\item If $h= j$ and $\ell_i \leq k_i$ for all $i$, then $\ell_i' \leq k'_i$ for all $i$;
\item If $0 \leq \ell_1 \leq \ldots \leq \ell_j$ and $$\ell_{i+1} \leq \ell_i + \ell_1 \text{ for all } i \in \{1, \ldots, j-1\},$$ then 
$$\ell'_{i+1} \leq  \ell'_i + \ell'_1 \text{ for all }i \in \{1, \ldots, j'-1\}$$ (and in particular $\ell'_i \leq i \cdot \ell'_1.$)
\end{enumerate}
\end{lemma}

\begin{proof}
We show that these properties hold when $s$ consists of just one turn, and for general $s$ the statement follows inductively. 
\begin{enumerate}
\item If the turn is winning, the new list is $(\ell_2, \ldots, \ell_{j - 1}) $ so the ordering remains. If it is losing, the new list is $(\ell_1, \ldots, \ell_j, \ell_{j} + \ell_1),$ and by assumption $\ell_j + \ell_1 \geq \ell_j.$
\item If the turn is winning, we just take out the first and last element from both lists, so $\ell_i' = \ell_{i + 1} \leq k_{i + 1} = k_i'$ for all $i$. If the turn is losing, the new lists are $$L' = (\ell_1, \ldots, \ell_j, \ell_{j} + \ell_1), \ K' = (k_1, \ldots, k_j, k_{j} + k_1),$$ and $\ell_j + \ell_1 \leq k_j + k_1,$ so the property holds again.
\item If the turn is winning, we have $\ell'_i = \ell_{i+1}, \ell'_1 = \ell_2 \geq \ell_1$, and hence $\ell'_{i +1} - \ell'_i \leq \ell_1 \leq \ell_2 = \ell'_1.$ If the turn is losing, we append $\ell_{j} + \ell_1$ to the end of the list, so clearly $\ell'_{j + 1} = \ell_j + \ell_1 = \ell'_j + \ell'_1.$
\end{enumerate} 
\end{proof}

\begin{defn}
We define $S_m$ to be the set of all sequences with exactly $m$ wins and $2m$ losses that are playable on the starting list $L = (1, 2)$, and we let $$s^*_m := \underwithbrace{m  \ \textbf{LLW}s }{\textbf{LLWLLW}\cdots\textbf{ LLW }}\in S_m.$$ 
\end{defn}

\vspace{0.2in}
\begin{defn}
Let $s = s_1\ldots s_k$ with $s_i \in \{\textbf{W}, \textbf{L}\}$ be a sequence playable on $(1, 2)$ with exactly $m$ wins (but not necessarily $2m$ losses), and let $i_1 < i_2 < \ldots < i_m$ be the indices of the wins in the sequence. For $1 \leq t \leq m$, we define $f_t(s)$ to be the  first number on the updated list once the prefix $s_1\ldots s_{i_t}$ has been played out, and we let $f_0(s):= 1$.
\end{defn}

\begin{lemma}\label{starprops}
$$f_t(s^*_m) \ll \pphi^t,$$
where $\pphi = \frac{1 + \sqrt{5}}{2}$ is the golden ratio. 
\end{lemma}
\begin{proof}
Observe that applying the sequence \textbf{LLW} to the list $(x, y)$ results in the list $(y, x + y)$:
$$(x, y) \mapsto (x, y, x + y) \mapsto (x, y, x + y, 2x + y) \mapsto (y, x + y).$$ Hence, 
$$f_{t + 1}(s^*_m) = f_{t}(s^*_m) + f_{t - 1}(s^*_m).$$ Since $f_0(s^*_m) = 1$, $f_1(s^*_m) = 2$, we see that $f_t(s^*_m) = F_{t + 2}$, where $F_k$ is the $k^\text{th}$ Fibonacci number. Thus, by Binet's formula,
$$f_t({s^*_m}) = \frac{\pphi^{t + 2} - (-\pphi)^{-t - 2}}{\sqrt{5}} \ll \pphi^t$$ as aimed.

\end{proof}
We now show that $s^*_m$ is ``maximal" in $S_m$ in the following sense:
\begin{lemma}\label{smallestnumber}
Let $m \geq 1$, $s \in S_m$ and $0 \leq t \leq m$. Then:
$$f_t(s) \leq f_t({s^*_m}).$$
\end{lemma}
\begin{proof}
We prove this statement by induction on $t$.
\begin{itemize}
\item For all $s \in S_m$, $$f_0(s) =f_0(s^*_m) = 1.$$ 
\item Since we only append numbers to the end of the list, right after the first win, no matter when it occurs, $2$ will become the first number on the list, so 
$$f_1(s) = f_1(s^*_m) = 2$$ for all $s$. 
\item Suppose for $1 \leq t < m$, we know that 
$$f_t(s) \leq f_t(s^*_m) \text{ and }f_{t-1}(s) \leq f_{t-1}(s^*_m).$$ Recall from the proof Lemma \ref{starprops} that 
$$f_{t + 1}(s^*_m) = f_{t}(s^*_m) + f_{t - 1}(s^*_m).$$ We claim that for any other $s \in S_m$, 
$$f_{t + 1}(s) \leq f_{t}(s) + f_{t - 1}(s)$$ (which would conclude the inductive step). Let $L = (f_{t-1}(s), f_t(s), x_1, \ldots, x_k)$ be the state of the list right before the $t^\text{th}$ winning turn. Note that $k \geq 1$ because we assumed the $t^\text{th}$ turn doesn't end the game (i.e.\ $t < m$). By Lemma \ref{basics}.$3$, we know that $$x_1\leq f_{t - 1}(s) + f_t(s).$$ There are two possibilities: 
\begin{enumerate}
\item If $k \geq 2$, after the $t^\text{th}$ win the list becomes $(f_t(s), x_1, \ldots, x_{k - 1})$, and hence 
$$f_{t + 1}(s) = x_1 \leq f_t(s) + f_{t - 1}(s).$$ 
\item If $k = 1$, then after the $t^\text{th}$ win the list becomes $(f_t(s))$ and then after the (forced) loss on the next turn the list becomes $(f_t(s), f_t(s))$, so 
$$f_{t + 1}(s) = f_t(s) \leq f_t(s) + f_{t-1}(s).$$ 
\end{enumerate}
This completes the inductive step. \vspace{-0.2in}
\end{itemize}
\end{proof}

Collecting the results above together gives us the following Theorem:
\begin{theorem}\label{Mnbound}
Let $L = (1, 2)$ and let $N = 3m + 1$ or $N = 3m$ be the stopping time. Then: 
$$M \ll m \pphi^m.$$
\end{theorem}
\begin{proof}
First, suppose $N = 3m + 1$. It is easy to see that, in such a game, there have to be exactly $2m$ losses and $m + 1$ wins. Since the last turn has to be winning, the sequence $s$ consisting of the first $3m$ turns of this game is in $S_m$.

Let $T$ denote the maximal number appearing in any of the lists that happen throughout this game (or, equivalently, during the first $3m$ turns of the game). Note that $M \leq 2T$, so it is equivalent to prove that $T \ll m \pphi^m.$

Now, note that since there are a total of $2m$ losses, the length of the list at any point of the game cannot exceed $2m+2$. Moreover, due to Lemma \ref{basics}, if $f$ is the first number on the list of length $j \leq 2m + 2$, then the largest number on the list is bounded from above by $jf \leq (2m + 2)f$. Hence by Lemma \ref{smallestnumber}, 
$$T \leq \sup_{0 \leq t \leq m} (2m + 2) f_t(s) \ll m f_t(s^*_m) \ll m \pphi^m,$$ where the last inequality follows from Lemma \ref{starprops}.

Similarly, for $N = 3m $, we know that there have to be $2m-1$ losses and $m + 1$ wins. Excluding the last winning turn, we get a sequence $s$ of $2m - 1$ losses and $m$ wins: $$s = s_1 \ldots s_{3m - 1}, s_i \in \{\text{\textbf{W},\textbf{L}}\}.$$ Now, note that 
$$s\textbf{L} = s_1\ldots s_{3m - 1}\textbf{L} \in S_m,$$ and so by Lemma \ref{smallestnumber},  $f_t(s\textbf{L}) \leq f_t(s^*_m)$. But appending a loss at the end doesn't change the set of the first numbers on the list, so $f_t(s) = f_t(s\textbf{L}) \leq f_t(s^*)$.  The rest of the proof carries over from the $N = 3m + 1$ case.

\end{proof}

\section{Expectation of the maximal bet size}
\hspace{.22in} In this section we collect the results from the two previous sections to prove Theorem \ref{main}.
We fist prove this for the list $L = (1, 2)$ and then show that it follows for any starting list.

\begin{lemma}\label{mainlemma}
The statement of Theorem \ref{main} holds for the games starting on the list $L = (1, 2)$. 
\end{lemma}
\begin{proof}
Let $M_N$ denote the largest number that can be bet in a game of length $N$ starting with the list $(1, 2)$. By Theorem \ref{Mnbound}, we know that 
$$M_N \ll N \pphi^{N/3},$$ where $\pphi$ is the golden ratio. Hence, by Corollary \ref{lengthtwo}, 
\begin{align*}
E(M) &\leq \sum_{n\geq 1} P(N = n)M_n \\
&\ll \sum_{m\geq 1}\frac{1}{2m +1} {3m \choose m} p^{m + 1} q^{2m-1} m\pphi^m+  \sum_{m\geq 0} \frac{1}{m + 1}{3m +1\choose m} p^{m + 1} q^{2m} m \pphi^m
\end{align*}
Now, using Stirling's approximation, 
$${3m \choose m} \ll \frac{\sqrt{m} \left( \frac{3m}{e}\right)^{3m}}{m \left( \frac{2m}{e}\right)^{2m}\left( \frac{m}{e}\right)^{m}}= \frac{3^{3m}}{2^{2m} \sqrt{m}}$$ and similarly 
$${3m + 1 \choose m} \ll \frac{3^{3m}}{2^{2m}\sqrt{m}}$$ as well. Hence, 
\begin{align*}
E[M] &\ll \sum_{m\geq 0}\frac{1}{2m +1} \frac{3^{3m}}{2^{2m}\sqrt{m}} p^{m + 1} q^{2m-1} m\pphi^m+  \sum_{m\geq 0} \frac{1}{m + 1}\frac{3^{3m}}{2^{2m}\sqrt{m}}p^{m + 1} q^{2m} m \pphi^m\\
&\ll \sum_{m\geq 0} \left(\frac{27 \pphi }{4}\right)^m  (pq^2)^m.
\end{align*}
This sum converges if and only if
 $$pq^2 < \frac{4}{27\pphi}.$$
Solving this cubic equation for $p \in [1/3, 1]$ yields 
$$p \in (c, 1], \text{ where }c \approx 0.613763$$ 
as aimed. 
\end{proof}

We now generalize this result for any other list. 

\begin{theorem}
Fix $p \in [1/3, 1]$ and suppose for the games starting on the list $(1, 2)$ we have that $E[M] < \infty.$ Then for any other starting list $L = (\ell_1, \ldots, \ell_j)$, $E[M] < \infty.$

\end{theorem}
\begin{proof}
For a list $L$, let $E_L[M]$ denote the expectation of the largest bet for a game with the starting list $L$. Fix $j \geq 2$ and consider a game played on $(1, 2)$ which begins with $j - 2$ consecutive losses (which happens with probability $q^{j - 2}$.) After those losses are played out, $(1, 2)$ is updated to $X= (1, 2, 3, \ldots, j)$. Now, note that 
$$E_{(1, 2)}[M] \geq q^{j - 2} E_{X}[M],$$ which by assumption implies that 
$$E_{X}[M] < \infty.$$ 
Next, consider any list $L = (\ell_1, \ldots, \ell_j)$ of length $j$. Let $c$ be a constant such that
$$c i > \ell_i, i \in \{1, \ldots, j\}.$$ For the list $cX := (c, 2c,\ldots, jc )$, we clearly have 
$$E_{cX}[M] = c E_{X}[M] < \infty.$$ Hence, by Lemma \ref{basics} part $2$, it follows that 
$$E_{L}[M] \leq E_{c X}[M] < \infty $$ as well. This proves the claim for $j \geq 2$. 

Lastly, for a list $(c)$ of length $1$, note that we have $$pq E_{(2)}[M] \leq E_{(1, 2)}[M] < \infty$$ and hence 
$$E_{(c)}[M] = (c/2) E_{(2)}[M] < \infty$$

\end{proof}

This concludes the proof of Theorem \ref{main}.

\section{Acknowledgments}

I want to thank Persi Diaconis for teaching a wonderful class that this paper came out of, Stewart Ethier for kindly sharing his thoughts and notes on the topic, and Thomas Church, Levent Alpoge and Stewart Ethier for very helpful comments on this paper. 

\bibliography{Labouchere}
\bibliographystyle{plain}

\end{document}